\documentclass[11pt]{amsart}
\usepackage{amssymb,amsmath,amscd,url}

\newtheorem{theorem}{Theorem}[section] 
\newtheorem{lemma}[theorem]{Lemma}

\theoremstyle{remark}

\theoremstyle{definition}

\usepackage{amsfonts,url}

\newcommand{\q}{/\!/}

\newcommand{\nothing}[1]{\relax}

\newcommand{\Z}{\mathbb{Z}}
\newcommand{\C}{\mathbb{C}}

\newcommand{\fX}{\mathfrak{X}}

\def\SL{{\rm SL}}
\def\GL{{\rm GL}}

\def\fT{{\mathfrak T}}

\def\val{{\rm val}}
\def\Sym{{\rm Sym}}

\def\rec{{\rm rec}}

\def\SL{{\rm SL}}

\def\GL{{\rm GL}}

\def\Art{{\rm Art}}

\def\End{{\rm End}}

\def\Sp{{\rm Sp}}

\def\rec{{\rm rec}}
\def\Langlands{{\rm Langlands}}
\def\Irr{{\rm Irr}}
\def\WDRep{{\cG}}

\def\Gal{{\rm Gal}}

\def\Irr{{\rm Irr}}

\def\vol{{\rm vol}}

\def\cG{{\mathcal G}}

\def\Irr{{\mathbf {Irr}}}

\def\fX{{\mathfrak X}}

\def\fS{{\mathfrak S}}

\def\Sp{{\rm Sp}}

\def\q{{/\!/}}

\title[Epsilon factors]{Epsilon factors as algebraic characters on the smooth dual of $\GL_n$}
\author{Roger Plymen}
\address{School of Mathematics, Manchester University, Manchester M13 9PL, UK}
\email{roger.j.plymen@manchester.ac.uk}
\date{\today}

\begin{document}
\begin{abstract}
Let $K$ be a non-archimedean local field and let $G = \GL_n(K)$.   We have shown in previous work that the smooth dual $\Irr(G)$ admits a complex structure: 
%it is the disjoint union of smooth algebraic varieties, each of which is the quotient of a complex torus by a product of symmetric groups.   
in this article we show how the epsilon factors interface with this complex structure. 
The epsilon factors, up to a constant term, factor as invariant characters through the corresponding complex tori.   For the arithmetically unramified smooth dual of $\GL_n$, 
%the smooth varieties form a single extended quotient, namely $T\q W$ where $T$ is a maximal torus in the complex Langlands dual $\GL_n(\C)$, and $W$ is the Weyl group.   In this case, 
 we provide explicit formulas for the invariant characters.   
\end{abstract}

\maketitle

\section{Introduction}    Let $K$ be a non-archimedean local field and let $G = \GL_n(K)$.  In the spirit of the Langlands program, this short paper provides a link between number theory and geometry.   
The arithmetic comes from $\varepsilon$-factors, which are naturally associated to representations of Galois groups, or more generally to representations of the Weil-Deligne group $WD_K$ of $K$. 

 The smooth dual $\Irr(G)$ possesses a nice geometry: it is the disjoint union of smooth complex algebraic varieties, each of which is the quotient of a complex torus by a product of symmetric groups,
 see \cite{BP}.      
 
 In this article we show how the epsilon factors  interface with this complex structure.

 We continue with some background on the epsilon factors \cite{L}, \cite{D}. The epsilon factors are very important and central in the theory of Artin  $L$-functions. If $E$ denotes a global field, 
then we have the absolute Weil group $W_E$, see \cite[6.2.6]{MP}.  The completed $L$-function $L (s,V )$
 of a representation $W_E \to \GL(V )$ of the Weil group $W_E$ of the global field $E$ 
defines a meromorphic function in the complex plane satisfying the functional equation
\[
 L(s,V) = \varepsilon(s,V)L(1 - s,V^*)
 \]
where $V^*$ is the dual of the representation $V$ of $W_E$, and the epsilon factor $\varepsilon(s,V)$  
 is defined by the product
 \begin{align}\label{artin}
 \varepsilon(s,V) = \prod \varepsilon_{E_{\nu}}(s, V_{\nu}, \psi_{\nu}).
 \end{align}
Here, $\psi_{\nu}$ is the local component at a place $\nu$ of a non-trivial additive character $\psi$ of $\mathbb{A}_E /E$.  Then $\psi_{\nu}$ is a
non-trivial additive character of the local field $E_{\nu}$, see \cite[5.11]{D}.   As usual, $\mathbb{A}_E$ denotes the adeles of $E$.   

From now on,  we will focus on the non-archimedean places of the global field $E$, and we will write $K$ for the non-archimedean local field  $E_{\nu}$.  An elementary substitution, see \S2, allows on to replace the epsilon factor $\varepsilon_K(s,V,\psi)$ with three variables, by the epsilon factor  $\varepsilon_K(V,\psi)$ with two variables.   From now on, we will be concerned with the epsilon factor $\varepsilon_K(V,\psi)$.

From the point of view of the local Langlands correspondence for $\GL_n$, the relevant representations are the Weil-Deligne representations, see section 2.   
The set $\WDRep_n(K)$ of equivalence classes of $n$-dimensional Weil-Deligne representations can be organised as a disjoint union of complex algebraic
varieties:
\[
\WDRep_n(K) = \bigsqcup \fX
%\mathcal{O}(\rho')
\]
Each variety $\fX$ arises in the following way.   Let $\rho'$ denote a Weil-Deligne representation of the Weil group $W_K$, and let $m$ denote the number of 
indecomposable summands in $\rho'$.  Then $\fX$ is the quotient of a complex torus $\fT$ of dimension $m$ by a certain finite group $\fS$:
\begin{align}\label{torus}
\fX = \fT/\fS.
\end{align}
 
 By a \emph{rational character}, or \emph{algebraic character}, or simply \emph{character}, we shall mean a morphism of algebraic groups
 \[
 \fT \to \C^\times.
 \]
 Such a character has the form 
 \[
 (z_1, \ldots, z_m) \mapsto z_1^{\beta_1} \cdots z_m^{\beta_m}
 \]
  where the $\beta_j$ are all integers. 
  
  \begin{theorem}\label{main}  Up to a constant $e(\fX,\psi)$, each epsilon factor $\varepsilon_K(V,\psi)$ factors through a rational character $\chi(\fX,\psi)$ of $\fT$.   Quite specifically, we have  
 \begin{align}\label{map}
 (z_1, \ldots, z_k) \mapsto (z_1^{\beta_1}, \ldots, z_k^{\beta_k})
 \end{align}
 where  the $z_j$ are torus coordinates, $\fX$ is the orbit of the Weil-Deligne representation
 \[
  V_1 \otimes \Sp(d_1) \oplus \cdots \oplus V_k \otimes \Sp(d_k)
 \]
 and
 \[
 \beta_j = (d_j - 1) \dim V^I_j  + d_j[a(V_j) + n(\psi) \dim(V_j)]
 \]
 where $V_1, \ldots, V_k$ are irreducible representations of the local Weil group $W_K$, 
 $a(V_j)$ denotes the Artin conductor exponent of $V_j$,  $n(\psi)$ denotes the conductor of $\psi$, and $I$ denotes the inertia subgroup of 
 $W_K$.  
  \end{theorem}  
  
   Since the conductors are integers, 
 the number $\beta_j$ is an integer.   So the map (\ref{map}) is a rational character of $\fT$.
 
 %We have set  \[ z_j = q_K^{ - s_j} = \omega_{s_j}(\Phi_K) \]
 
 The rather lengthy formula for the constant $e(\fX,\psi)$ appears later in this article, see (\ref{e}) and (\ref{ee}).   
 
 We emphasize that the character $\chi(\fX,\psi)$ and the constant $e(\fX,\psi)$ \emph{depend only on the connected component} $\fX$, once a base point has been chosen in 
 $\fX$, and  the 
 additive character $\psi$ has been chosen and fixed.

In \S 4, we focus on a part of the smooth dual of $\GL_n(K)$, namely the arithmetically unramified smooth dual.   The extended quotient $T\q W$ is a model for this part of the dual,
where $T$ is a maximal torus in the Langlands dual group $\GL_n(\C)$ and $W$ is the Weyl group of $\GL_n(\C)$.  We calculate explicitly the epsilon factors.
  \medskip

 This article is an expanded account of a talk given at the conference \emph{Geometry, representation theory and the Baum-Connes conjecture}, Fields Institute, July 2016, which was 
 part of the EU Quantum Dynamics network activities.
  We thank  Paul Baum for several valuable conversations, which led to major changes in the exposition of this article.   
  
  We take this opportunity to thank Paul Baum for an inspiring collaboration, 
  which has withstood the test of time, and which has enabled us to discover and explore several new mathematical vistas.   Paul observed that the article  \cite{BP} 
   was the tip of a large iceberg.   Further exploration led to several articles, of which \cite{ABPS} is the most recent.

In writing this article, we were greatly influenced by the preprint of Ikeda \cite{Ikeda}. The main background reference is Deligne \cite{D}, but we prefer to use the notation in Langlands\cite{L}.  
We thank the referee for providing many  insightful comments which influenced the final version of this Note.     

\section{Weil-Deligne representations} We need to recall some material, following closely the exposition in \cite{BP}.   
Let $K$ be a non-archimedean local field.   The normalized valuation will be denoted $\val_K$, the norm of $x \in K$ will be denoted $|| x||_K$, 
and the cardinality of the residue field of $K$ will be denoted $q_K$.        We have the local Artin reciprocity map
\[
\Art_K : W_K \to K^\times. 
\] 
We will write
\[
||w|| = ||\Art_K(w)||_K = q_K^{- \val_K(\Art_K(w)}
\]
for all $w \in W_K$.   

The homomorphism $d : W_K \to \Z$ is defined by 
\[
d(w) = \val_K(\Art_K(w)).
\]
  The Weil group $W_K$ fits into a short exact sequence 
\[
0 \to I_K \to W_K \overset{d} \longrightarrow  \Z \to 0
\]
 where $I_K$ is the inertia group of $K$.     
%The Weil-Deligne group is defined as $WD_K = W_K \ltimes \C$ where $w \cdot z = ||w|| z$.   
A \emph{Weil-Deligne representation} is a pair $(\rho, N)$ consisting of a continuous representation $\rho : W_K \to \GL_n(V)$,
 $\dim_{\C}(V) = n$, together with a nilpotent endomorphism $N \in \End(V)$ such that 
 \[
 \rho(w)N \rho(w)^{-1} = ||w|| N.   
 \]  

For any $n \geq 1$, the representation $\Sp(n)$ is defined by 
\[
V = \C^n = \C e_0 + \cdots + \C e_{n-1}
\]
with $\rho(w) e_i = ||w||^i e_i$ and $Ne_i = e_{i + 1} \, (0 \leq i \leq n -1), \, Ne_{n-1} = 0$.   

Let $\WDRep_n(K)$ be the set of equivalence classes of semisimple $n$-dimensional Weil-Deligne representations.     
Let $\Irr(\GL_n(K))$ be the set of equivalence classes of irreducible smooth representations of $\GL_n(K)$.   

We recall the local Langlands correspondence
\[
\rec_K : \Irr(\GL_n(K)) \to \WDRep_n(K)
\]
which is unique subject to the conditions listed in \cite[p.2]{HT}.  See also the single trace condition of Scholze in \cite[Theorem 1.2(a)]{Sch}.

We identify the elements of the set $\WDRep_1(K)$, the quasicharacters of $W_K$, with quasicharacters of $K^\times$ via the local Artin reciprocity map
$\Art_K$.   The local Langlands correspondence is compatible with twisting by quasicharacters \cite[p.2]{HT}.

A quasicharacter $\eta : W_K \to \C^\times$ is (arithmetically) \emph{unramified} if $\eta$ is trivial on the inertia group $I_K$.   In that case we have 
$\eta(w) = z^{d(w)}$ with $z \in \C^\times$.    The group of unramified quasicharacters of $W_K$ is denoted $\Psi(W_K)$.   Let $\Phi = \Phi_K$ denote a geometric
Frobenius element in $W_K$.    The isomorphism $\Psi(W_K) \simeq \C^\times$ is secured by the map $\eta \mapsto \eta(\Phi_K)$.

Let now 
\[
\rho' = \rho_1 \otimes \Sp(d_1) \oplus \cdots \oplus \rho_m \otimes \Sp(d_m)
\]
 be a Weil-Deligne representation.  In this formula, $\rho_1, \ldots, \rho_m$ are irreducible representations of the Weil group $W_K$.    The set
\[
\{\eta_1 \rho_1 \otimes \Sp(d_1) \oplus \cdots \oplus \eta_m \rho_m \otimes \Sp(d_m) : \eta_1, \ldots , \eta_m \in \Psi(W_K) \}
\]
will be called the \emph{orbit} of $\rho'$ under the action of 
\[
\Psi(W_K) \times \cdots \times \Psi(W_K)
\]
($m$ factors).   This orbit will be denoted $\fX = \fX(\rho')$.   The orbits create a partition of $\WDRep_n(K)$.   
The set  $\WDRep_n(K)$ is a disjoint union of orbits:
\[
\WDRep_n(K) = \bigsqcup \fX
\]
We note that $\Psi(W_K)^m \simeq (\C^\times)^m$, a complex torus.  

Each irreducible representation $\rho$ of $W_K$ has a \emph{torsion number}: the order of the cyclic group 
of all those unramified characters $\eta$ of $W_K$ for which $\rho \otimes \eta \cong \rho$.   The torsion number of $\rho_i$ will be denoted $\tau_i$.    
To determine the structure of each orbit, we have to pay attention to the torsion numbers of $\rho_1, \ldots , \rho_m$ and to the action of $\GL_n(\C)$ by conjugation.

Let $\mu_{\tau} \subset \C$ denote the cyclic group of order $\tau$.
The orbit $\mathcal{O}(\rho')$ will have the structure of the \emph{quotient}  torus
\begin{align}\label{prod}
\Psi(W_K) \times \cdots \times \Psi(W_K)/ (\mu_{\tau_1}  \times \cdots \times \mu_{\tau_m})
\end{align}
modulo a product of symmetric groups.   We will write $z_1, \ldots, z_m$ for coordinates on this quotient torus. To be precise, we are writing $z_j$ as a coset:
\begin{align}\label{zzz}
z_j = \eta_j(\Phi_K)\cdot \mu_{\tau_j}
\end{align}

In this way, the set $\WDRep_n(K)$ acquires 
(locally) the structure of complex algebraic variety.   Each connected component in  $\WDRep_n(K)$ is the quotient of a 
complex torus by a product of symmetric groups.

%We shall view each orbit $\mathcal{O}(\rho')$ as a pointed set, by choosing a Galois representative for each irreducible representation of $W_K$.  
We recall that, given an irreducible representation $V$ of $W_K$, there exists an irreducible representation $V^{\Gal}$ of Galois type such that $V = V^{\Gal} \otimes \omega_s$ for some
$s \in \C$, see \cite[(2.2.1)]{Tate}. 

This allows us to  view the orbit $\fX(\rho')$ 
as a pointed complex algebraic variety, with base point 
\[
\rho' = \rho_1 \otimes \Sp(d_1) \oplus \cdots \oplus \rho_m \otimes \Sp(d_m).
\]
where each  $\rho_j$ is of Galois type.

\section{The formulas}   The elementary substitution referred to in the Introduction is as follows.  Let
 \[
 \varepsilon_K(s,V, \psi) = \varepsilon_K(V \otimes \omega_{s - 1/2}, \psi)
 \]
  for all $s \in \C$, see \cite[(3.6.4)]{Tate}, \cite[p.6]{L}.  For $s \in \C$, $\omega_s : W_K \to \C^\times$ is the unramified quasicharacter defined by
  $\omega_s(w) = ||w||_K^s$ for all $w \in W_K$.   To compare this quasicharacter with those in \S 2, note the following:
  \begin{align*}
  \omega_s(w) & = ||w||_K^s\\
  & = q_K^{- s \cdot \val_K(\Art_K(w))}\\
  & = z^{d(w)}  
  \end{align*}
  with $z = q_K^{-s} \in \C^\times$.

  If $V$ is a $1$-dimensional continuous complex representation of $W_K$, and $\chi : W_K \to \C^\times$ is the corresponding quasicharacter, then 
 $\varepsilon_K(\chi, \psi)$
   is the abelian local constant of Tate, see \cite[(3.6.3)]{Tate}.  
   
    We note that $\varepsilon_K(V,\psi)$ is denoted $\varepsilon_K^{\Langlands}(V,\psi)$ in \cite{Ikeda} and $\varepsilon_L(V,\psi)$
in \cite[3.6]{Tate}.

We recall that, if $(V,N)$ is any $\Phi$-semisimple Weil-Deligne representation, then we have a finite direct sum decomposition of $(V,N)$ into indecomposable 
Weil-Deligne representations as follows:
\begin{align}\label{D}
(V,N) = V_1 \otimes \Sp(d_1) \oplus \cdots \oplus V_m \otimes \Sp(d_m)
\end{align}

We  will write 
\[
V_j = V_j^{\Gal} \otimes \omega_{s_j}.
\]

\begin{lemma}\label{artin} Let $a(V)$ denote the Artin conductor exponent of $V$.   Then we have $a(V \otimes \omega_s) = a(V)$.   
\end{lemma}
\begin{proof} The definition is
\begin{align}\label{cond}
a(V) = \dim V - \dim V^I + \sum_{k \geq 1} \frac{1}{[I:I_k]} \cdot \dim V/ V^{I_k}
\end{align}
where $I = I_0 \supset I_1 \supset \ldots \supset I_k \supset \ldots$ are the ramification subgroups of the inertia group $I$.  
 
We have 
\[\dim  (V \otimes \omega_s) = \dim  V
\]
Now $\omega_s$ is an unramified quasi-character of $W_K$:
\[
\omega_s(I) = ||\Art_K(I)||^s = ||U_K||^s = 1
\]
 and so
\[
(V \otimes \omega_s)^{I_k} = V^{I_k}
\]
for all $k \geq 0$.       The result now follows from (\ref{cond}).   
\end{proof}

In particular, we have
\[
a(V_j) =  a(V_j^{\Gal}).
\]

We need the following three items in order to compute epsilon factors.  

\subsection{Additivity}  Additivity with respect to $V$, see \cite[3.4.2]{Tate}, \cite[Theorem A (ii)]{L}:
\begin{align}\label{add}
\varepsilon_K(V_1 \oplus \cdots \oplus V_k,\psi) = \varepsilon_K(V_1,\psi) \ldots \varepsilon_K(V_k,\psi)
\end{align}

\subsection{Unramified twist}  Behaviour under unramified twist, see \cite[3.4.5]{Tate}, \cite[Lemma 22.4]{L}:
\begin{align}\label{twist}
\varepsilon_K(V \otimes \omega_s,\psi) = \varepsilon_K(V,\psi) q^{-s[a(V) + n(\psi) \dim V]}
\end{align}
where $a(V)$ is the Artin conductor exponent of $V$, and $n(\psi)$ is the conductor of $\psi$.   

\subsection{The extension formula}  The extension to Weil-Deligne representations is as follows \cite[4.1.6]{Tate}:
\begin{align}\label{ext}
\varepsilon_K((V,N), \psi): = \varepsilon_K(V,\psi) \det( - \Phi |V^I / V^I_N)
\end{align}

\subsection{The term $\varepsilon_K(V,\psi)$}  A typical direct summand of (\ref{D}) as a representation of $W_K$ is
\[
V_j^{\Gal} \otimes \omega_{s_j} \otimes \omega_k
\]
with $1 \leq j \leq m, \; 0 \leq k \leq d_j - 1$.  We have
\[
V_j^{\Gal} \otimes \omega_{s_j} \otimes \omega_k   = V_j^{\Gal} \otimes \omega_{s_j + k} 
\]
 For this summand, we have  by (\ref{twist})
\begin{align}\label{typical}
\varepsilon_K(V_j^{\Gal} \otimes \omega_{s_j} \otimes \omega_k,\psi)   =  \varepsilon_K(V_j^{\Gal},\psi) q^{-(s_j + k)[a(V_j^{\Gal}) + n(\psi) \dim V_j^{\Gal}]}
\end{align}
and then the formula for $\varepsilon_K(V,\psi)$ follows from (\ref{add}).  Applying Lemma \ref{artin}, we obtain
\begin{align}\label{1}
\varepsilon_K(V,\psi) = \prod_{j=1}^m\left(\varepsilon_K(V_j^{\Gal},\psi)\right)^{d_j} \cdot q^{-[s_j d_j  + \frac{(d_j-1)d_j}{2}][a(V_j) + n(\psi)\dim(V_j)]}
\end{align}

Note that Ikeda succeeds in describing the numbers $\varepsilon_K(V_j^{\Gal},\psi)$ in terms of the non-abelian local class field theory of $K$,  see 
\cite[Theorem 5.4]{Ikeda}.

\subsection{The determinant}  The determinant is additive 
\[
\det(A \oplus B) = (\det A)(\det B)
\]
 and so it suffices to consider a typical factor in (\ref{ext}), namely
\begin{align} \label{3}
\det(-\Phi | E_j^I / (E_j^I)_{N_j})
 \end{align}
 where $E_j$ is the $W_K$-module given by
 \begin{align*}\label{E}
 E_j & = V_j \otimes  (\omega_0 \oplus \omega_1 \oplus \cdots \oplus  \omega_{d_j -1})\\
 & =  (V_j ^{\Gal} \otimes \omega_{s_j})\otimes  (\omega_0 \oplus \omega_1 \oplus \cdots \oplus  \omega_{d_j -1})\\
& = V_j ^{\Gal} \otimes (\omega_{s_j} \oplus  \omega_{s_j + 1} \oplus \cdots \oplus  \omega_{s_j +d_j -1})
 %V_j \otimes \Sp(d_j)\\
 %&  = V_j \otimes (\C e_0  \oplus \cdots \oplus \C e_{d_j-1})
 \end{align*}

 We note that $V_j ^I  = (V_j^{\Gal})^I$.  Then the $W_K$-submodule fixed by the inertia group $I$ is
 \[
 E_j^I :   = V_j^I \otimes (\omega_{s_j} \oplus \omega_{ s_j +1}\oplus \cdots \oplus \omega_{s_j + d_j -1})
 \]

  The $W_K$-submodule of $E_j$ annihilated by $N_j$ is 
 \[
 (E_j)_{N_j}  =  V_j^{\Gal} \otimes \omega_{s_j + d_j-1}
 \]
 from which it follows that
  \[
 (E_j)^I_{N_j}  =  V_j^I \otimes \omega_{s_j + d_j-1}
 \]
 For the quotient we have the following $W_K$-module:
 \begin{align}\label{E}
 E_j^I / (E_j)^I_{N_j} \simeq V_j^I \otimes (\omega_{s_j} \oplus \cdots \oplus \omega_{s_j + d_j-2})
% & \simeq V^I_j \otimes (\omega_0 \oplus \omega _1 \oplus  \cdots \oplus \omega_{d_j-2})
 \end{align}

     Recall that
\[
\omega_s(\Phi) = ||\varpi_K||^s = q_K^{-s}
\]
 
 It is enough to compute the action of $- \Phi$ on the $W_K$-module $V_j^I \otimes \omega_{s_j} \otimes \omega_k$ with $0 \leq k \leq d_j -1$.  
On this $W_K$-module, $- \Phi$ will act as 
\[
q^{ - (s_j + k)} (- \Phi | V_j^I )
\]
and the determinant will be 
\[
q^{ - (s_j + k) \dim V_j^I} \det (- \Phi | V_j^I)
\]

%Now $V_j$ is a $W_K$-module which we will denote as $\rho_j : W_K \to \GL(V_j)$.  
%A typical direct summand in (\ref{E}) is $V_j^I \otimes \omega_k$ with $0 \leq k \leq d_j -2$.  We have to compute the determinant of 
%$(\omega_{s_j} \otimes \rho_j \otimes \omega_k)( - \Phi)$ acting on the vector space $V_j^I$. This determinant will be
%$q^{- (s_j +k) \dim V_j^I} \cdot \det( - \rho_j(\Phi)|V_j^I)$.  
There are $d_j -1$ direct summands in (\ref{E}) so the resulting determinant will be the product
\begin{align}\label{det}
& \prod_{k = 0}^{d_j -2}q^{- (s_j +k) \dim V_j^I} \cdot \det( - \rho_j(\Phi)|V_j^I) \\
& =  \det( - \rho_j(\Phi) | V_j^I )^{d_j - 1} \cdot q^{-s_j (d_j-1)\dim V_j^I }\cdot q^{ - (1 + 2 + \ldots + d_j-2) \dim V_j^I}\\
& =  \det( - \rho_j(\Phi) | V_j^I )^{d_j-1}  \cdot q^{-s_j (d_j-1)\dim V_j^I }\cdot q^{ - \frac{1}{2}(d_j-2)(d_j-1)\dim V_j^I}
\end{align}
provided that $d_j \geq 3$.   By inspection, this formula is also valid for $d_j = 1$ or $2$.

\subsection{The term $\varepsilon_K((V,N)),\psi)$}  We recall the discussion in \S 2 of the quotient torus (\ref{prod}), especially the definition  (\ref{zzz}) of the torus coordinates $z_1, \ldots, z_m$:  
\begin{align*}
z_j & = \omega_{s_j}(\Phi_K) \cdot \mu_{\tau_j}\\
& = q_K^{- s_j} \cdot \mu_{\tau_j}
\end{align*}
where $\tau_j$ is the torsion number of $V_j^{\Gal}$ and $\mu_{\tau_j} \subset \C$ is the cyclic group of order $\tau_j$.   

 From the extension formula (\ref{ext}) we  infer that  $\varepsilon((V,N,\psi)$ is the product of (\ref{1}) and (\ref{det}).  
This product is of the form
\begin{align}\label{exp}
const \cdot z_1^{\beta_1} \cdots z_m^{\beta_m}
\end{align}
where   
\[
\beta_j = (d_j-1) \dim V_j^I  + d_j[a(V_j) + n(\psi)\dim(V_j) ]
\] 
for all $1 \leq j \leq m$. Note that $\beta_j$ is an \emph{integer}:
\[
\beta_j \in \Z.
\]

The constant can be read off from (\ref{1}) and (\ref{det}).   
   Apart from the constant term, the formula (\ref{exp}) for the epsilon factor is a \emph{rational character} of the 
quotient torus (a complex torus of dimension $m$): 
\[
(z_1, \ldots, z_m) \mapsto z_1^{\beta_1} \cdots z_m^{\beta_m}
\]  

Consider the following set:
\begin{align}\label{WD}
 \{\omega_{s_1} \otimes V^{\Gal}_1 \otimes \Sp(d_1) \oplus \cdots \oplus \omega_{s_m} \otimes V^{\Gal}_m \otimes \Sp(d_m) : s_1, \ldots, s_m \in \C\}.
\end{align}
After allowing for conjugacy in the Langlands dual group $\GL_n(\C)$, this set has the structure of a complex algebraic variety $\fX$ in $\cG_n$.  In fact $\fX$ is a connected
component in $\cG_n(K)$:
\[
\fX \subset \cG_n(K)
\]
Applying the local Langlands correspondence, we have, by transport of structure, a connected component in the smooth dual:
\[
\rec_K^{-1} (\fX) \subset \Irr(\GL_n(K))
\]

We emphasize that this imposes a topology on $\Irr(\GL_n(K))$ which is finer than the standard (Zariski) topology in representation theory.   Our topology on the smooth dual 
is well-adapted to the study of the Hecke algebra $\mathcal{H}(\GL_n(K))$ (see Theorem 3 in \cite{BP}), and is, indeed, well-adapted to the study of epsilon factors, 
the subject of this article.

Looking carefully at the formulas (\ref{1}) and (\ref{det}), we see that the constant in (\ref{exp}) depends on the variety $\fX$, the choice of base point of $\fX$, 
and the additive character $\psi$.  We will
denote this constant by $e(\fX, \psi)$, so that (\ref{exp}) can be re-written
\begin{align}\label{exp2}
e(\fX,\psi) \cdot z_1^{\beta_1} \cdots z_m^{\beta_m}
\end{align}
which, up to the constant $e(\fX,\psi)$, factors as a rational character through $\fT$, in the notation of (\ref{torus}).  
   The constant
 $e(\fX,\psi)$ is itself the product of
  \begin{align}\label{e}   
  \prod_{j=1}^m\left(\varepsilon_K(V_j^{\Gal},\psi)\right)^{d_j} \cdot q^{- \frac{1}{2}(d_j-1)d_j [a(V_j) + n(\psi)\dim(V_j)]}
 % \det( - \rho_j(\Phi) | V_j^I )^{d_j-1}  \cdot q^{ - \frac{1}{2}(d_j-2)(d_j-1)\dim V_j^I}
 \end{align}
 with 
 \begin{align}\label{ee}
 \prod_{j = 1}^m \det( - \rho_j(\Phi) | V_j^I )^{d_j-1}  \cdot q^{ - \frac{1}{2}(d_j-2)(d_j-1)\dim V_j^I}
 \end{align}
This establishes our main result, Theorem \ref{main}, which could perhaps be useful in the geometric Langlands program. 
  
The terms $\varepsilon_K(V_j^{\Gal}, \psi)$ are the epsilon factors attached to irreducible representations of the local absolute Galois group $G_K$.   
These terms are defined in \cite[p.15]{Ikeda}.   There is one case where they are readily computed.

\begin{lemma}\label{111}  Let $\psi$ be an additive character $K \to \C^\times$.  Then we have $\varepsilon_K(1,\psi) = 1$.
\end{lemma}
\begin{proof} We start with the classical formula in \cite[3.6.3]{Tate}:
\[
\varepsilon_K(\chi, \psi) =  \chi(c) \frac{\int_{\mathcal{O}^{\times}} \chi^{-1}(u) \psi(u/c) \,du}{| \int_{\mathcal{O}^{\times}} \chi^{-1}(u) \psi(u/c) \,du|}
\]
where $c$ is an element of $K^\times$ of valuation $a(\chi) + n(\psi)$.  Now set $\chi = 1$ and let $n(\psi) = k$.   Then we take $c = \varpi^k$.   Then 
$u \in \mathcal{O}^\times \implies u/c \in \varpi^{-k} \mathcal{O}^\times$. But we have
$\psi(\varpi^{-k}\mathcal{O}) = 1$ since $\psi$ has conductor $k$.   Therefore we have
\begin{align*}
\varepsilon_K(1, \psi)  &=   \frac{\int_{\mathcal{O}^{\times}}  \psi(u/c) \,du}{| \int_{\mathcal{O}^{\times}} \psi(u/c) \,du|}\\
&  = \frac{\vol(\mathcal{O}^\times)}{ |\vol(\mathcal{O}^\times) |}\\
&   = 1
\end{align*}

\end{proof}

\section{The arithmetically unramified representations of $\GL_n(K)$}    Here, the underlying representation of the Weil group is the trivial
$n$-dimensional representation $\rho : W_K \to \GL_n(\C)$.   So we have $V_j^{\Gal}  = 1, 1 \leq j \leq n$.  

%At this point, the Langlands dual intervenes.   So let $G = \GL_n(F)$ and let $G^\vee = \GL_n(\C)$ be the Langlands dual of $G$.   Let $T^\vee$ be the standard maximal torus in
%$G^\vee$. 
%$T$ be the standard maximal torus
%in $\GL_n(\C)$, let 
Let $W$ be the Weyl group $\fS_n$. The arithmetically unramified representations of $\GL_n(K)$ have, by definition, the following set of Langlands parameters 
(Weil-Deligne representations):
\begin{align}\label{LL}
\{\omega_{s_1} \otimes \Sp(d_1) \oplus \cdots \oplus \omega_{s_k} \otimes \Sp(d_k) : s_j \in \C\}
\end{align}
where $d_1 + \cdots + d_k = n$.  This set determines a complex algebraic variety $\fX$  in $\cG_n(K)$.

We choose $\psi$ to have conductor $0$, and now apply Lemma \ref{111}.  In this case  $\beta_j = d_j - 1$.   We have
\begin{align*}\label{Phi1}
\varepsilon_K((V,N,\psi) & = e(\fX,\psi) \prod_{j=1}^m q^{-(d_j-1)s_j} \\
& = e(\fX,\psi) \prod_{j=1}^m z_j^{d_j-1}
\end{align*}
where
\begin{align*}
e(\fX,\psi) = \prod_{j=1}^m (-1)^{d_j-1} q^{-(d_j-1)(d_j-2)/2}
\end{align*}
and $z_j: = q^{-s_j}$.

The epsilon factor records the dimensions $d_j$ of the special representations $\Sp(d_j)$ which occur in 
the Weil-Deligne representation $(V,N)$.

We will now re-organise the partition $d_1 + \cdots + d_k = n$.   
Suppose that this partition has distinct parts $t_1, \ldots , t_m$ with $t_1 <  t_2 < \cdots < t_m$ and that $t_j$ is repeated $r_j$ times so that 
\begin{equation}\label{part}
r_1t_1 + \cdots r_mt_m = n.
\end{equation}
Then, as a function on the complex torus $(\C^\times)^{r_1 + \cdots + r_m}$,  the epsilon factor is \emph{invariant under the following product of symmetric groups}:
\[
\fS_{r_1} \times \fS_{r_2} \times \cdots \times \fS_{r_m}
\]
and therefore factors through the following quotient variety
\[
(\C^\times)^{r_1} / \fS_{r_1} \times \cdots \times (\C^\times)^{r_m} / \fS_{r_m}
\]
%This variety is a twisted sector in the  extended quotient $T \q W$, as we now proceed to explain.    

 Let  $T$ denote the standard  maximal torus in the Langlands dual group $\GL_n(\C)$, and let $W$ be the Weyl group of $\GL_n(\C)$.
 We have the \emph{inertia space}
 \[
 \mathcal{I}_W(T): = \{(w,t): w \in W, t \in T,  wt = t\}
 \]
 Then $\mathcal{I}_W(T)$ admits an action of $W$, namely $ \alpha \cdot (w,t) = (\alpha w \alpha ^{-1}, \alpha \cdot t)$.   We define $T\q W$ to be the quotient: 
 %$\mathcal{I}_W(T)/W$ is the \emph{inertia stack}, see \cite{JM}.   We will write
 \[
 T\q W: = \mathcal{I}_W(T)/W
 \]
 This is the \emph{extended quotient} of $T$ by $W$, sometimes called the \emph{inertia stack}, see \cite[\S 1.1]{JM}.

Let $T^w$ denote the $w$-fixed set, and let $Z(w)$ be the $W$-centralizer of $w$.   Choose one $w$ in each $W$-conjugacy class, then we have
\[
T \q W = \bigsqcup T^w / Z(w).
\]
%The contributions to this sum, indexed by conjugacy classes other than the identity, are the \emph{twisted sectors}, see \cite[\S 1.3]{JM}.   

 The conjugacy classes in the symmetric group $W$ are in bijection with the partitions of $n$.   The partition in (\ref{part}) determines the 
 permutation comprising $r_1$ cycles of length $t_1$, \ldots, $r_m$ cycles of length $t_m$.   If this permutation is denoted $w$, then we have
 \[
 T^w  = (\C^\times)^{r_1}  \times \cdots \times (\C^\times)^{r_m} 
\]

The centralizer $Z(w)$ is a product of wreath products 
\[
 \Z/t_1\Z \wr \fS_{r_1} \times \cdots \times \Z/t_m\Z \wr \fS_{r_m}
\]
but the cyclic groups act trivially and so we have
\begin{align*}
 %T^w & = (\C^\times)^{r_1}  \times \cdots \times (\C^\times)^{r_m} \\
 T^w/Z(w) & = (\C^\times)^{r_1} / \fS_{r_1} \times \cdots \times (\C^\times)^{r_m} / \fS_{r_m}
\end{align*}

Every irreducible component in $T\q W$ is accounted for in this way.
The epsilon factors have precisely the amount of symmetry required  to factor through these quotient varieties.

\textsc{Example}.  Here, we consider the following Weil-Deligne representation of $\GL_{19}(K)$:
\[
\omega_{s_1}\Sp(2) \oplus \omega_{s_2} \Sp(2) \oplus \omega_{s_3} \Sp(2) \oplus \omega_{s_4} \Sp(3) \oplus \omega_{s_5} \Sp(3) 
\oplus \omega_{s_6} \Sp(7)
\]
The epsilon factor of this representation is 
\[
const \cdot z_1z_2z_3z_4^2z_5^2z_6^6
\]
which will factor through the following irreducible component of the extended quotient $T \q W$:
\[
\Sym^3(\C^\times) \times \Sym^2 (\C^\times) \times \C^\times
\]

This perfectly illustrates the symmetry properties of the epsilon factors.   Each epsilon factor has precisely the symmetry, \emph{neither more nor less}, of the corresponding 
irreducible component in the extended quotient $T \q W$.   Each epsilon factor will therefore factor through the corresponding irreducible component in $T \q W$.   

\section{The general case}   Let $G $ be a  reductive $p$-adic group and let $^L G$ be the $L$-group 
\[ 
{^L G} = G^\vee \rtimes W_K
\]
where $G^\vee$ is the complex dual group of $G$.      Let 
\[
r :  {^L G} \to \GL(V)
\]
be a representation of $^L G$ on the complex vector space $V$, as in \cite[2.6]{Bor}.  Let  $W'_K$ denote the Weil-Deligne group $W_K \ltimes \C$, defined by $w.z = ||w||z$. 
As in \cite[8.2]{Bor}, let $\phi$ be an $L$-parameter for $G$:
\[
\phi : W'_K \to {^L G}
\]
  
Following  \cite[12.1]{Bor}, we have the composite
\[
r \circ \phi : W'_K \to \GL(V)
\]
The representation $r \circ \phi$ determines, and is determined by, a Weil-Deligne representation $(V,N)$. 
So we obtain an epsilon factor $\varepsilon_K((V,N), \psi)$ to which we may apply all the preceding material.

\end{document}